\documentclass[12pt]{amsart}
\usepackage{amsmath}
\usepackage{amsfonts}
\usepackage{amssymb}
\usepackage{epsfig}
\usepackage{color}

\newtheorem{lm}{Lemma}[section]
\newtheorem{tr}{Theorem}[section]

\newtheorem{df}{Definition}[section]

\newtheorem{pr}{Proposition}[section]
\newtheorem{qs}{Question}[section]

\begin{document}
\author{F. Marko}
\title{Symmetrizers for Schur superalgebras}
\address{The Pennsylvania State University, 76 University Drive, Hazleton,
PA 18202, USA}
\email{fxm13@psu.edu}
\date{}\begin{abstract} 
For the Schur superalgebra $S=S(m|n,r)$ over a ground field $K$ of characteristic zero, we define symmetrizers $T^{\lambda}[i:j]$ of the ordered pairs of tableaux $T_i, T_j$ of the shape $\lambda$ and 
show that the $K$-span $A_{\lambda,K}$ of all symmetrizers $T^{\lambda}[i:j]$ has a basis consisting of $T^{\lambda}[i:j]$ for $T_i,T_j$ semistandard. The $S$-superbimodule $A_{\lambda,K}$ is identified as 
$D_{\lambda}\otimes_K D^o_{\lambda}$, where $D_\lambda$ and $D^o_\lambda$ are left and right irreducible $S$-supermodules of the highest weight $\lambda$.

We define modified symmetrizers $T^{\lambda}\{i:j\}$ and show that their $\mathbb{Z}$-span form a $\mathbb{Z}$-form $A_{\lambda,\mathbb{Z}}$ of $A_{\lambda, \mathbb{Q}}$. We show that every modified symmetrizer $T^\lambda\{i:j\}$ is a $\mathbb{Z}$-linear combination of symmetrizers $T^\lambda\{i:j\}$ for $T_i, T_j$ semistandard. Using modular reduction to a field $K$ of characteristic $p>2$, we obtain that $A_{\lambda,K}$ has a basis consisting of modified symmetrizers $T^\lambda\{i:j\}$ for $T_i, T_j$ semistandard.
\end{abstract}
\maketitle

\section{Introduction}

Let us start by describing the classical case of Schur algebras $S(m,r)$ over a field of an arbitrary characteristic.
The degree $r$ part $A(m,r)$ of the bialgebra of polynomials $A(m)$ has a filtration by $S(m,r)$-bimodules $A_{\leq \lambda}$, 
where $\lambda$ is a partition of $r$ of no more than $m$ parts.
The bimodule $A_{\leq \lambda}=\sum_{\zeta'\unlhd \lambda'} A_{\zeta}$, where $A_{\zeta}$ is the $K$-span of 
bideterminants $T^{\zeta}(i:j)$.
The sum $A_{<\lambda}=\sum_{\zeta'\lhd \lambda'} A_{\zeta}$ is also a bimodule, and the quotient 
$A_{\leq \lambda}/A_{<\lambda}$ is 
isomorphic to $\nabla(\lambda)\otimes\Delta(\lambda)^*$, 
where $\nabla(\lambda)$ is the left costandard module of the highest weight $\lambda$ (the Schur module)
and $\Delta(\lambda)^*$ is the right module which is the dual of the left module $\Delta(\lambda)$, the standard module of the highest weight $\lambda$ (the Weyl module).
The right module $\Delta(\lambda)^*$ is given as a span of bideterminants $T^{\lambda}(i:\ell)$, and the left module $\nabla(\lambda)$ is the span of bideterminants $T^{\lambda}(\ell:j)$, where $T^{\lambda}_\ell$ is the canonical tableau of the shape $\lambda$. For more details, consult \cite{gr,martin,doub,dkr,cl0,cl1,cl,brini}. 

The results for $S(m,r)$ translate readily to bimomodules over the general linear group $GL(m)$. Bimodules over
Schur algebras $S(m,r)$ are polynomial bimodules over $GL(m)$, and tensoring with a power of one-dimensional representation $Det$ yields the description of the Donkin-Koppinen filtration of the coordinate algebra $K[GL(m)]$.   
The factors of this filtrations are isomorphic to $H^0_{GL(m)}(\lambda)\otimes V_{GL(m)}(\lambda)^*$, where
$H^0_{GL(m)}(\lambda)$ is the left induced module of the highest weight $\lambda$ and $V_{GL(m)}(\lambda)$ is the left Weyl module of highest weight $\lambda$. For more details, see \cite{don,kop,lz2}.

In the case of characteristic zero, $S(m,r)$ and $GL(m)$ are semisimple. If the characteristic $p>0$, then $S(m,r)$ and 
$GL(m)$ are no longer semisimple. However, $S(m,r)$ is a quasi-hereditary algebra and the category of modules over $GL(m)$ is the highest weight category. For additional details, see \cite{martin, don, cps}.

In the characteristic zero case, the superalgebra $S=S(m|n,r)$ is semisimple, while $G=GL(m|n)$ is not. However, supermodules over $G$ form the highest weight category.  
In the case of positive characteristic $p>2$, modules over $G$ still form the highest weight category and there is Donkin-Koppinen filtration with superbimodule filtration factors isomorphic to 
$H^0_{G}(\lambda)\otimes V_G(\lambda)^*$, where $H^0_G(\lambda)$ and $V_G(\lambda)$ are the left induced and  Weyl supermodules of the highest weight $\lambda$, respectively. For more details, see \cite{kac,z,lz2}.

In the supercase, 
tensoring polynomial $S$-superbimodules with one-dimensional supermodule Berezian $Ber$ does not cover all rational $G$-superbimodules. Additionally, $S$ is no longer a quasi-hereditary algebra. Nevertheless, the $S$-superbimodule structure of $\nabla_S(\lambda)\otimes \Delta_S(\lambda)^*$ gives the $G$-supermodule structure 
of $H^0_G(\lambda)\otimes V_G(\lambda)^*$ in the case when $V_G(\lambda)$ and $H^0_G(\lambda)$ polynomial supermodules. According to \cite{m1}, this happens if and only if 
$\lambda_m\geq n$.

It was noted earlier that in the case of the Schur algebras, the bimodules $\nabla(\lambda)\otimes \Delta(\lambda)^*$ were described using bideterminants $T^{\lambda}(i:j)$. To transition to the superalgebra $S$, we define the symmetrizer of an ordered pair of tableaux $T_i$, $T_j$ of the
shape $\lambda$ by 
\begin{equation*}
T^{\lambda}[i:j]=\sum_{\rho\in R(T)} \sum_{\kappa\in C(T)} sgn(\kappa)
\chi_{i*\rho,j*\kappa}.
\end{equation*}
We can write $T^{\lambda}[i:j]$ equivalently as 
\begin{equation*}
\sum_{\rho\in R(T)} \sum_{\kappa\in C(T)}sgn(\kappa)
\chi_{i,j*\kappa*\rho}
\end{equation*}
or 
\begin{equation*}
\sum_{\kappa \in C(T)}\sum_{\rho\in R(T)} sgn(\sigma)
\chi_{i*\rho*\kappa,j}.
\end{equation*}

An equivalent form of the symmetrizers $T^{\lambda}[i:j]$ appeared first in \cite{doub} in connection to Gordan-Cappelli formula. In \cite{cl0}, the right symmetrizers $(S|\framebox{T})$ appear in the description of Schur modules for $GL(m)$,  and the left symmetrizers, denoted $(\framebox{S}|T)$, appear in the description of Weyl modules for $GL(m)$ over the ground field of arbitrary characteristic.
Later on, in the context of $GL(m|n)$, the symmetrizers $(\framebox{S}|T)$ are used in \cite{br} to describe covariant modules over the ground field of characteristic zero.

If we replace our definition of the symmetrizers with the dual definition 
\[\sum_{\kappa\in C(T)} \sum_{\rho\in R(T)} sgn(\kappa)
\chi_{i*\kappa,j*\rho},\]
then all results we obtain will naturally extend for this dual definition. We leave it to the reader to reformulate all results.

Denote by $D_{\lambda}$ the irreducible left $S$-module of the highest
weight $\lambda$ and by $D_{\lambda}^o$ the irreducible right $S$-module of
the highest weight $\lambda$, and denote by $A_{\lambda}=A_{\lambda,K}$ the $K$-span of all symmetrizers $%
T^{\lambda}[i:j]$ of the shape $\lambda$.

Assuming the characteristic of the ground field $K$ is zero and $\lambda$ is a $(m|n)$-hook partition, we show in Theorem \ref{t4.1} that symmetrizers
$T^{\lambda}[i:j]$, where $T_i, T_j$ are semistandard tableaux, form a basis of $A_{\lambda}$, and 
$A_{\lambda}\simeq D_{\lambda}\otimes D^o_{\lambda}$ as $S$-superbimodules.

In the second half of the paper, we define modified symmetrizers $T^{\lambda}\{i:j\}$ and in Theorem \ref{t5.1} we show that their $\mathbb{Z}$-span form a $\mathbb{Z}$-form
$A_{\lambda,\mathbb{Z}}$ of $A_{\lambda, \mathbb{Q}}$. 
In Theorem \ref{t5.2} we show that every modified symmetrizer $T^{\lambda}\{i:j\}$ is a $\mathbb{Z}$-linear combination of modified symmetrizers $T^{\lambda}\{i:j\}$ for $T_i, T_j$ semistandard.
Using a process of modular reduction, we obtain that over a field of characteristic $p>2$, the superbimodule $A_{\lambda, K}$ has a basis consisting of modified symmetrizers $T^\lambda \{i:j\}$. 
This result is related to our previous work \cite{mz}. 

Finally, we note that the isomorphism $A_{\lambda,K} \simeq \nabla_S(\lambda)\otimes \Delta_S(\lambda)^*$ is not true for all $(m|n)$-hook weight $\lambda$. However, if such isomorphism is valid for a hook weight $\lambda$ such that 
$\lambda_m\geq n$, then it describes the $G$-superbimodule structure of the factor of Donkin-Koppinen fitration corresponding to such $\lambda$.

\section{Preliminaries}
We mostly adhere to notation from \cite{mu}. 
A $K$-space $V$ is a superspace when it is $\mathbb{Z}_2$-graded, where $V=V_0\oplus V_1$ and the parity $|v_0|=0$ is even for $v_0\in V_0$, and $|v_1|=1$ is odd for $v_1\in V_1$. 
We define the parity of symbols $1\leq i\leq m$ to be even, and the parity of symbols $m+1\leq j\leq m+n$ to be odd.

\subsection{Tableaux}

Throughout the paper, we asssume that $\lambda=(\lambda_1, \ldots, \lambda_k)$ is a $(m|n)$-hook partition of $r$, which means $\lambda_{m+1}\leq n$. 
According to \cite{br}, $(m,n)$-hook partititions $\lambda$ are in one-to-one correspondence with irreducible modules over the Schur superalgebra $S(m|n,r)$. Denote by $\unlhd$ the dominance order on partitions of $r$.

Denote by $\Sigma_r$ the permutation group on $r$ elements.
Fix a basic tableau $T^{\lambda}=T$ of the shape $\lambda$, that is a bijection from the diagram $[\lambda]$ to the set $\{1, \ldots, r\}$,  and denote by $C(T)$ and $R(T)$, respectively, the subgroup of $\Sigma_r$
that permutes columns of $T$, and rows of $T$, respectively. 

Unless otherwise stated, we assume that entries of tableaux belong to the alphabet $\{1, \ldots, m+n\}$, where the symbols $1, \ldots, m$ are even, and symbols $m+1, \ldots, m+n$ are odd.  (We will extend this alphabet in \ref{4.1}.) Under this assumption, each tableau is given as $T_i$
for a uniquely defined multiindex $i$ of length $r$.

A tableau $T_i$ is called semistandard if its entries are weakly increasing in each column top to bottom, and in each row left to right. Additionally, we require that no even symbol is repeated in any collumn, and no odd symbol is repeated in any row of $T_i$. It is obvious that semistandard tableaux exist only for partitions $\lambda$ that are $(m|n)$-hook partitions.

For a tableau $T_i$, denote by $c_{qp}(i)$ the number of occurences of symbols not exceeding $p$ in the first $q$ columns
of $T_i$, and by $r_{qp}(i)$ the number of occurences of symbols not exceeding $p$ in the first $q$ rows of $T_i$. 

We lists the data sets $(c_{qp}(i))$ and $(r_{qp}(i))$ with respect to the lexicographic order $\leq_{lex}$, first listing them by $q$, then by $p$. We define the column dominance preorder $\unlhd_c$ by 
$T_i\unlhd_c T_j$ if and only if $(c_{qp}(i))\leq_{lex} (c_{qp}(j))$, and the row dominance preorder $\unlhd_r$ by
$T_i\unlhd_r T_j$ if and only if $(r_{qp}(i))\leq_{lex} (r_{qp}(j))$.

For later use, denote by $T_{ev}$ the subtableau of $T$ consisting of the first $m$ rows of $T$, and by $T_{odd}$ the skew tableau $T\setminus T_{ev}$.
Denote by $T_{\ell}$ such tableau of shape $\lambda$ for which the entries in the $i$th row of $T_{\ell, ev}$ equals $i$,
and the entries in the $j$th column of $T_{\ell, odd}$ equals $m+j$.

Also denote $\ell(\lambda)$, tableau $T$ and its even part $_eT$ and odd part $_oT$.

\subsection{Superbialgebra $A(m|n)$}

Denote by $C$ a generic $(m+n)\times(m+n)$-matrix given as $C=(c_{ij})$, 
where the parity $|c_{ij}|=|i|+|j|$. Thus $C$ can be considered as a block matrix
\[\begin{pmatrix} C_{00}&C_{01}\\C_{10}&C_{11}\end{pmatrix},\] 
where $C_{00}$ is an $m\times m$ matrix, and $C_{11}$ is an $n\times n$ matrix, both with even entries, while
$C_{01}$ is an $m\times n$ matrix and $C_{10}$ is an $n\times m$ matrix, both with odd entries.

The superbialgebra $A(m|n)$ is generated by variables $c_{ij}$ subject to the supercommutativity relation
$c_{ij}c_{kl}=(-1)^{|c_{ij}||c_{kl}|}c_{kl}c_{ij}$, comultiplication $\Delta(c_{ij})=\sum_k c_{ik}\otimes c_{kj}$ 
and counit $\epsilon(c_{ij})=\delta_{ij}$. It has a natural grading by degrees $r$ and its component in degree $r$ is a supercoalgebra denoted by $A(m|n,r)$.

Let $E$ be the standard $S$-supermodule with the basis $e_i$ for $1\leq i \leq m + n$ and the coaction 
$\tau_E (e_i) =\sum_{1\leq k \leq m+n} e_k \otimes  c_{ki}$.
The superspace $E^{\otimes r}$ has a basis consisting of elements $e_i = e_{i_1}\otimes \ldots e_{i_r}$
for $i\in I(m|n,r)$, where 
$I(m|n,r)$ denotes the set of all maps from the set $\{1, \ldots , r\}$ to $\{1, \ldots, m+n\}$.
The coaction is given as $\tau_{E^{\otimes r}}(e_i)=\sum_{j\in I(m|n,r)} e_j \otimes \chi_{ji}.$

The functions $\chi_{i,j}$ are given as
\[\chi_{i,j}=(-1)^{\sum_{t=1}^r deg(i_t)(deg(i_{t+1})+deg(j_{t+1})+\ldots + deg(i_r)+deg(j_r))} c_{i,j}.\]
The symmetric group $\Sigma_r$ acts on multiindex $j\in I(m|n,r)$ naturally by permutation of coordinates. The $*$ action of $\Sigma_r$ on 
$j$, denoted $j*\sigma$ for $\sigma\in \Sigma_r$ is given as
$j*\sigma=(-1)^{s(j,\sigma)}j\sigma$, where
\[s(j,\pi)=\#|\{(a,b):1\leq a<b\leq r; \pi(a)>\pi(b), i_a,i_b>m\}|.\]

For a pair $(i,j)$, define $a_{k,l}=\#|\{1\leq t\leq r|i_t=k, j_t=l\}|$, 
\[\bar{\chi}_{i,j}=c_{11}^{a_{11}}\ldots c_{1,m+n}^{a_{1,m+n}}c_{21}^{a_{21}}\ldots c_{m+n,m+n}^{a_{m+n,m+n}},\]
$I^+=\{(i,j)|\chi_{i,j}=\bar{\chi}_{i,j}\}$ and $I^-=\{(i,j)|\chi_{i,j}=-\bar{\chi}_{i,j}\}$

We define the equivalence relation $i\sim k$ if and only if $i\pi=k$ for some $\pi\in \Sigma_r$.
Also define $(i,j)\sim(k,l)$  if and only if there is $\pi\in \Sigma_r$ such that 
$i\pi=k$ and $j\pi=l$.

Then $\chi_{i,j}=\pm \chi_{k,l}$ if and only if $(i,j)\sim (k,l)$.

\subsection{Schur superalgebra $S(m|n,r)$}

The superalgebra $S=S(m|n,r)$ is the dual of the supercoalgebra $A(m|n,r)$.
Generators $\xi_{i,j}\in S(m|n,r)$ are defined via bilinear pairing $\langle,\rangle$ by 
\[\begin{aligned}&\langle \chi_{k,l},\xi_{i,j}\rangle = \xi_{i,j}(\chi_{k,l}) = (-1)^{s(i,\pi)+s(j,\pi)} \text{ if } & (i,j)\sim (k,l) \text{ such that }\\
&& i\pi=k, j\pi=l;\\
&\langle \chi_{k,l},\xi_{i,j}\rangle=\xi_{i,j}(\chi_{k,l})=0 \text{ otherwise. }
\end{aligned}\]

Thus 
\[\begin{aligned}&\langle \chi_{k,l},\xi_{i,j}\rangle=\xi_{i,j}(\chi_{k,l})=1 &\text{ if } (k,l)\in I^+ \text{ and } (i,j)\sim (k,l);\\
&\langle \chi_{k,l},\xi_{i,j}\rangle=\xi_{i,j}(\chi_{k,l})=-1 &\text{ if } (k,l)\in I^- \text{ and } (i,j)\sim (k,l); \\ 
&\langle \chi_{k,l},\xi_{i,j}\rangle=\xi_{i,j}(\chi_{k,l})=0 &\text{ otherwise.}
\end{aligned}\]

Then $\xi_{i,j}=\pm \xi_{k,l}$ if and only if $(i,j)\sim (k,l)$.
The product of $\xi_{i,j}$ and $\xi_{k,l}$ is given as
\[\xi_{ij}\xi_{k,l}=\sum_{p,q} Z(i,j,k,l,p,q) \xi_{p,q},\]
where $Z(i,j,k,l,p,q)\neq 0$ implies $(i,j)\sim (p,s)$ and $(s,q)\sim(k,l)$.
For a detailed description of $Z(i,j,k,l,p,q)$, see \cite{mu}.

In particular, $\xi_{i,j}\xi_{k,l}\neq 0$ implies $j\sim k$.

Finally, 
\[\langle \chi_{p,q},\xi_{i,j}\xi_{k,l}\rangle=\sum_{s} \langle \chi_{p,s},\xi_{i,j}\rangle\langle \chi_{s,q},\xi_{k,l}\rangle.\]

The costandard supermodule $\nabla_S(\lambda)$ is the largest $S$-supersubmodule of the (left) injective supermodule $I_S(\lambda)$ such that all of its simple composition factors $L_S(\mu)$ satisfy $\mu\unlhd \lambda$. 
The standard supermodule $\Delta_S(\lambda)$ is the largest $S$-superfactormodule of the (left) projective supermodule 
$P_S(\lambda)$ such such that all of its simple composition factors $L_S(\mu)$ satisfy $\mu\unlhd \lambda$. 
The supermodule $\nabla(\lambda)$ has a simple socle isomorphic to $L_S(\lambda)$ and $\Delta_S(\lambda)$ has a simple top isomorphic to $L_S(\lambda)$.

\section{Reduction to semistandard tableaux}

Since the action of $\xi_{u,v}\in S$ is given as 
\begin{equation*}
\xi_{u,v}\chi_{i,j}=\sum_{a\in I(m|n,r)} \xi_{u,v}(\chi_{a,j})\chi_{i,a},
\end{equation*}
\begin{equation*}
\chi_{i,j}\xi_{u,v}=\sum_{b\in I(m|n,r)} \chi_{u,v}(\chi_{i,b})\chi_{b,j},
\end{equation*}
the action of $\xi_{i,j}$ on symmetrizers is given as 
\begin{equation*}
\xi_{u,v}T^{\lambda}[i:j]=\sum_{a\in I(m|n,r)}
\xi_{u,v}(\chi_{a,j})T^{\lambda}[i:a],
\end{equation*}
\begin{equation*}
T^{\lambda}[i:j]\xi_{u,v}=\sum_{b\in I(m|n,r)}
\chi_{u,v}(\chi_{i,b})T^{\lambda}[b:j].
\end{equation*}

Therefore, $A_{\lambda}$ is a $S$-superbimodule.

\subsection{$T_i$ fixed}

Assume first that $T_{i}$ is fixed. In this case we work with the definition
\[T[i:j]=\sum_{\rho\in R(T)} \sum_{\kappa\in C(T)}sgn(\kappa)
\chi_{i,j*\kappa*\rho}.\]

The following lemma is analogous to Lemma 4.3 of \cite{mz} and page 123 of 
\cite{mu}.

\begin{lm}
\label{l1} If $\sigma\in C(T)$, then $T[i:j\ast \sigma
]=sgn(\sigma)T[i:j]$. Thus, if $T_{j}$ has two identical even entries in the same column, then $T[i:j]=0$.
\end{lm}

\begin{proof}
Write 
\[\begin{aligned}&T[i:j\ast \sigma]=
\sum_{\rho \in R(T)}\sum_{\kappa \in C(T)}sgn(\kappa )\chi _{i ,j\ast \sigma \ast \kappa \ast \rho }\\
&=sgn(\sigma) \sum_{\rho \in R(T)}\sum_{\bar{\kappa} \in C(T)}sgn(\bar{\kappa} )
\chi _{i ,j\ast \bar{\kappa} \ast \rho }=sgn(\sigma)T[i:j],\end{aligned}\]
where we set $\bar{\kappa}=\sigma\kappa$.

If $T_j$ has two identical even entries in the same column, choose $\sigma\in C(T)$ 
to be any transposition interchanging identical even entries. Then $j \ast \sigma=j$ and 
$T[i:j]=T[i:j\ast \sigma]=-T[i:j]$, showing $T[i:j]=0$.   
\end{proof}

The following lemma provides the Garnir relation - compare it with Lemma 4.4. of \cite{mz} and (2.4.1a) of \cite{mu}.

\begin{lm}
\label{l2} Let $X$ be a subset of $k$-th column of $T$, $Y$ be a subset of $%
k+1$-st column of $T$ such that the cardinality of the set $X\cup Y$ exceed
the length of the $k$-th column of $T$. Let $\{\sigma_1, \ldots, \sigma_l\}$
be a left transversal of $S_{X}\times S_{Y}$ in $S_{X\cup Y}$, Then $%
\sum_{t=1}^l sgn(\sigma_t)T[i:j*\sigma_t]=0$.
\end{lm}
\begin{proof}
Each element $\pi$ in $B=S_{X\cup Y}C(T)$ can be written uniquely as $\pi=\sigma_t\kappa$ for $\sigma_t$ as above and $\kappa\in C(T)$.
Therefore
\[\sum_{t=1}^l sgn(\sigma_t)T[i:j*\sigma_t]=\sum_{\rho \in R(T)}\sum_{\pi\in B} sgn(\pi) \chi_{i,j*\pi*\rho}.\]
By Peel's theorem (see \cite{peel} or (2.4.1a) of \cite{mu}), the set $B$ is a disjoint union of pairs $\{\pi, \pi\alpha_{\pi}\}$, 
where $\pi\in B$ and transposition $\alpha_{\pi}\in R(T)$.

Since $\alpha_{\pi}\in R(T)$, we have 
\[\sum_{\rho \in R(T)} \chi_{i,j*\pi\alpha_{\pi}*\rho} = \sum_{\bar{\rho} \in R(T)} \chi_{i,j*\pi*\bar{\rho}},\] 
where $\bar{\rho}=\alpha_{\pi}\rho$. Since $sgn(\pi\alpha_{\pi})=-sgn(\pi)$, we obtain 
\[\sum_{\rho\in R(T)} sgn(\pi) \chi_{i,j*\pi*\rho} + \sum_{\rho\in R(T)} sgn(\pi\alpha_{\pi}) \chi_{i,j*\pi\alpha_{\pi}*\rho} =0.\]
\end{proof}

Consider the column lexicographic order on the entries of the tableau $T$, listing the entries by columns first and then ordering entries in each row from top to bottom. 
Using this order, we induce the order on tableux $T_k$ of the same  shape as follows. The tableau $T_j<T_k$ if and only if the first entry, where $T_j$ and $T_k$ differ, is smaller in $T_j$ than in $T_k$.
A tableau $T_k$ is minimal in this order if and only if it is semistandard, or one of its rows contains two identical odd entries.

A reduction to semistandard tableaux is finalized with the help of Carter-Lusztig theorem (p. 214 of \cite{cl} or (2.4.1b) of \cite{mu}), compare it 
to Proposition 4.5 of \cite{mz}.

\begin{pr}\label{p3.1}
\label{p1} Every $T[i:j]$ is a $\mathbb{Q}$-linear combination of $T[i:l]$
for $T_l$ semistandard such that $T_l\unlhd_r T_j$.
\end{pr}
\begin{proof}
We proceed by induction on the order of tableaux $T_l$ of the same content defined above.

We can assume that entries in columns of $T_j$ are weakly increasing from top to bottom. Using Lemma \ref{l1} we can also assume that there are no identical even entries in any column of $T_j$. 

If $T_j$ is minimal but not semistandard, then we will see later that $T[i:j]=0$. This will verify the base step of the induction.

For the inductive step, assume that $T_j$ is not semistandard and consider the first position (in column lexicographic order) where this condition is violated. Depicted below is a part of the labeling of the basic tableau $T$ consisting of the entries in its $d$-th and $d+1$-st columns $C_d$ and $C_{d+1}$, together with the two cases for corresponding entries in $T_i$. In the first case $j_{a_q}>j_{b_q}$ is first such inequality; and in the second case $j_{a_q}=j_{b_q}>m$ is first such equality.

\[T:
\begin{array}{ccc}
C_d&&C_{d+1}\\
&&\\
a_1&&b_1\\
\ldots&&\ldots\\
\ldots&&b_{p-1}\\
\ldots&&b_p\\
\ldots&&\ldots\\
a_q&&b_q\\
\ldots&&\ldots\\
a_r&&\\
a_{r+1}&&\\
\ldots\\
a_s&&
\end{array}
\qquad
T_j:
\begin{array}{ccc}
C_d&&C_{d+1}\\
&&\\
j_{a_1}&&j_{b_1}\\
\ldots&&\ldots\\
\ldots&&j_{b_{p-1}}\\
\ldots&&j_{b_p}\\
\ldots&&\ldots\\
j_{a_q}&>&j_{b_q}\\
\ldots&&\ldots\\
j_{a_r}&&\\
j_{a_{r+1}}&&\\
\ldots&&\\
j_{a_s}&&
\end{array}\]
or 
\[ T_j:
\begin{array}{ccc}
C_d&&C_{d+1}\\
&&\\
j_{a_1}&&j_{b_1}\\
\ldots&&\ldots\\
\ldots&&j_{b_{p-1}}\\
&&\wedge\\
\ldots&&j_{b_p}\\
&&\vert\vert\\
\ldots&&\ldots\\
&&\vert\vert\\
j_{a_q}&=&j_{b_q}\\
\vert\vert&&\\
\ldots&&\ldots\\
\vert\vert&&\\
j_{a_r}&&\\
\wedge&&\\
j_{a_{r+1}}&&\\
\ldots&&\\
j_{a_s}&&
\end{array}
\]

We consider two cases depending on which condition is violated first (that is either $j_{a_q}>j_{b_q}$ or
$j_{a_q}=j_{b_q}$.)

Let $X=\{a_q \ldots, a_s\}$ be a subset of the $d$-th column of $T$, 
$Y=\{b_1, \ldots, b_q\}$ be a subset of the $d+1$-st column of $T$, and let 
$\{\sigma_1=1, \ldots, \sigma_l\}$ be a left transversal of $S_{X}\times S_{Y}$ in $S_{X\cup Y}$.

Case \underline{$j_{a_q}>j_{b_q}$:}

By Lemma \ref{l2}
we obtain 
\[T[i:j]=-\sum_{t=2}^l sgn(\sigma_t)T[i:j*\sigma_t].\]
Since $j_{b_1}\leq \ldots \leq j_{b_q}<j_{a_q}\leq \ldots \leq j_{a_s}$ and each $T_{j\sigma_t}$, for $t>1$, has one of the entries $j_{b_1}, \ldots, j_{b_q}$ in its $d$-th column, we infer that $T_{j\sigma_t}<T_j$ for each $t>1$.
Additionally, $T_{j\sigma_t}\unlhd_r T_j$.

Case \underline{$j_{a_q}=j_{b_q}>m$:}

By Lemma \ref{l2}
we obtain 
\[\sum_{t=1}^l sgn(\sigma_t)T[i:j*\sigma_t]=0.\]

Since 
\[j_{b_1}\leq \ldots \leq j_{b_{p-1}}=j_{b_p}=\ldots =j_{b_q}=j_{a_q}=\ldots = j_{a_r}<j_{a{r+1}}\leq \ldots \leq j_{a_s},\]
there are $\binom{r-p+2}{r-q+1}=\frac{(r-p+2)!}{(q-p+1)!(r-q+1)!}$ elements $\sigma_t$ such that $j=j\sigma_t$. Those $\sigma_t$ only permute the identical odd elements and thus $j*\sigma_t=sgn(\sigma_t)j$, and 
$sgn(\sigma_t)T[i:j*\sigma_t]=T[i:j]$.
If $\sigma_t$ is such that $j\sigma_t\neq j$, then the $d$-th column of $T_{j\sigma_t}$ contains one of the entries
$j_{b_1}, \ldots, j_{b_{p-1}}$, which implies $T_{j\sigma_t}<T_j$.
Additionally, $T_{j\sigma_t}\unlhd_r T_j$.

Therefore
\[T[i:j]=-\frac{1}{\binom{r-p+2}{r-q+1}}\sum_{j\sigma_t\neq j} sgn(\sigma_t)T[i:j*\sigma_t].\]
This concludes the inductive step of our argument.

If $T_j$ is minimal but not semistandard, then it contains two identical odd entries in the same row.
If $j\sigma_t\neq j$ for some $\sigma_t$, then $T_{j\sigma_t}<T_j$ contradicting the minimality of $T_j$.
Therefore $j\sigma_t=j$ for every $\sigma_t$ which implies $T[i:j]=0$.
\end{proof}

\subsection{$T_j$ fixed}

Now assume that $T_j$ is fixed. In this case we work with the definition
\[T[i:j]=\sum_{\kappa \in C(T)}\sum_{\rho\in R(T)} sgn(\kappa)
\chi_{i*\rho*\kappa,j}.\]

\begin{lm}
\label{l3} If $\sigma\in R(T)$, then $T[i*\sigma:j]=T[i:j]$. Thus, if $T_i$
has two identical odd entries in the same
row, then $T[i:j]=0$.
\end{lm}

\begin{proof}
Write 
\[\begin{aligned}&T[i*\sigma :j]=
\sum_{\kappa \in C(T)}\sum_{\rho \in R(T)} sgn(\kappa )\chi _{i*\sigma*\rho*\kappa ,j}\\
&= \sum_{\kappa \in C(T)}\sum_{\bar{\rho} \in R(T)}sgn(\kappa )
\chi _{i*\bar{\rho} *\kappa,j }=T[i:j],\end{aligned}\]
where we set $\bar{\rho}=\sigma\rho$.

If $T_j$ has two identical odd entries in the same row, choose $\sigma\in R(T)$ 
to be any transposition interchanging identical odd entries. Then $i* \sigma=-i$ and 
$T[i:j]=T[i*\sigma:j]=-T[i:j]$, showing $T[i:j]=0$.   
\end{proof}

\begin{lm}
\label{l4} Let $X$ be a subset of $k$-th row of $T$, $Y$ be a subset of $k+1$%
-st row of $T$ such that the cardinality of the set $X\cup Y$ exceed the
length of the $k$-th row of $T$. Let $\{\sigma_1, \ldots, \sigma_l\}$ be a
left transversal of $S_{X}\times S_{Y}$ in $S_{X\cup Y}$, Then $\sum_{t=1}^l
T[i*\sigma_t:j]=0$.
\end{lm}

\begin{proof}
Apply the transposition that changes $\lambda$ to $\lambda'$,  the basic tableau $T$ to $T'$,
and $T_i$ to $T'_i$ in such a way that $C(T)$ corresponds to $R(T')$, and $R(T)$ corresponds to $C(T')$. 
Under this transposition, $X'$ is a subset of the $k$-th colunmn of $T'$, $Y'$ is a subset of $k+1$-st column of $T'$
and the cardinality of $X'\cup Y'$ exceeds the length of the $k$-th column of $T'$. 
The right transversal $\{\sigma_1, \ldots, \sigma_t\}$ of $S_{X}\times X_{Y}$ in $S_{X\cup Y}$ is transposed to 
a left transversal $\{\sigma'_1, \ldots, \sigma'_t\}$ of $S_{X'}\times S_{Y'}$ in $S_{X'\cup Y'}$.

Each element $\pi'$ in $B'=S_{X'\cup Y'}C(T')$ can be written uniquely as $\pi'=\sigma'_t\kappa'$ for $\sigma'_t$ as above and $\kappa'\in C(T')$.
By Peel's theorem, the set $B'$ is a disjoint union of pairs $\{\pi', \pi'\alpha'_{\pi'}\}$, 
where $\pi'\in B'$ and transposition $\alpha'_{\pi'}\in R(T')$.
Thus $B=S_{X\cup Y}R(T)$ is a disjoint union of pairs $\{\pi, \pi\alpha_{\pi}\}$, where $\pi\in B$ and $\alpha_{\pi}\in C(T)$.

Therefore
\[\sum_{t=1}^l T[i*\sigma_t:j]=\sum_{\kappa \in C(T)}\sum_{\pi\in B} sgn(\kappa) \chi_{i*\pi*\kappa,j}.\]

Since $\alpha_{\pi}\in C(T)$ and $sgn(\kappa)=-sgn(\alpha_{\pi}\kappa)$, we have 
\[\sum_{\kappa \in C(T)}sgn(\kappa)\chi_{i*\pi\alpha_{\pi}*\kappa,j} = 
-\sum_{\bar{\kappa} \in C(T)} sgn(\bar{\kappa})\chi_{i*\pi*\bar{\kappa},j},\] 
where $\bar{\kappa}=\alpha_{\pi}\kappa$. Therefore,
\[\sum_{\kappa\in C(T)} sgn(\kappa) \chi_{i*\pi*\kappa,j} + \sum_{\kappa\in C(T)} sgn(\kappa) \chi_{i*\pi\alpha_{\pi}*\kappa,j} =0.\]
\end{proof}

\begin{pr}\label{p3.2}
\label{p2} Every $T[i:j]$ is a $\mathbb{Q}$-linear combination of $T[k:j]$ 
for $T_k$ semistandard such that $T_k\unlhd_c T_i$.
\end{pr}

\begin{proof}
The proof is analogous to that of Proposition \ref{p1}. We need to consider row lexicographic order on tableux $T_i$, and  adjacent rows $R_d$ and $R_{d+1}$ instead of coluns $C_d$ and $C_{d+1}$. We apply
Lemmae \ref{l3} and \ref{l4} instead of Lemmae \ref{l1} and \ref{l2} and observe that $\sigma_t$ permuting identical even entries satisfies $i*\sigma_t=i$.
In the case $i_{a_q}>i_{i_{b_q}}$ we get an identity
\[T[i:j]=-\sum_{t=2}^l T[i*\sigma_t:j],\]
where $T_{i\sigma_t}\unlhd_c T_i$.

In the case $i_{a_q}=i_{b_q}<m$ we get an identity
\[T[i:j]=-\frac{1}{\binom{r-p+2}{r-q+1}}\sum_{i\sigma_t\neq i} T[i*\sigma_t:j],\]
where $T_{i\sigma_t}\unlhd_c T_i$.
\end{proof}

\begin{tr}\label{t3.1}
Every symmetrizer $T^{\lambda}[i:j]$ is a $\mathbb{Q}$-linear combination of symmetrizers $T^{\lambda}[k:l]$, 
where $T_i,T_j$ are semistandard, $T_k\unlhd_c T_i$ and $T_l\unlhd_r T_j$.
\end{tr}
\begin{proof}
The statement follows from Propositions \ref{p3.1} and \ref{p3.2}.
\end{proof}

As noted earlier, if $\lambda$ is not a $(m|n)$-hook partition, then there are no semistandard tableaux of shape $\lambda$. This implies that all symmetrizers $T^{\lambda}[k:l]$ vanish for non $(m|n)$-hook partitions.

\section{Capelli operators}

We first define Capelli operators for our setting by adapting the notation from \cite{cl2}.

\subsection{Definition and properties of Capelli operators}\label{4.1}

We extend our alphabet that included even symbols $1, \ldots, m$ and odd symbols $m+1, \ldots, m+n$
by new colored symbols $\underline{1}, \ldots, \underline{k}, \ldots$ that are even, and new colored symbols 
$\overline{1}, \ldots, \overline{l}, \ldots$ that are odd. 
We denote by $\tilde{A}$ the polynomial ring $K[c_{uv}]$ in indeterminates $c_{uv}$, where $u,v\in \{1,\ldots, m+n\}
\cup\{\underline{k}|1\leq k\}\cup \{ \overline{l}|1\leq l\}$.

Assume the uncolored letter $i$ appears exactly $s$ times in the tableau $T_k$.
The map $D^r_L(\overline{j},i)$ (and $D^r_P(\underline{j},i)$, respectively) sends a tableau $T_k$ to a sum 
$\sum_{t=1}^{\binom{s}{r}} T_{k_t}$ of tableaux, where each $T_{k_t}$ is obtained from $T_k$ by replacing a subset of $r$ symbols $i$ in $T_k$ by $r$ colored symbols $\overline{j}$ (and $\underline{j}$, respectively).
It is understood that $D^r_L(\overline{j},i)(T_k)=0$ and $(T_k)D^r_P(\underline{j},i)=0$ whenever $s<r$, and also that 
$D^0_L(\overline{j},i)$ and $D^0_P(\underline{j},i)$ are identity maps.

The map $D^r_L(\overline{j},i)$ induces the (letter) polarization operator that sends the monomial $\chi_{kl}$ to the sum $\sum_{t=1}^{\binom{s}{r}} \chi_{k_tl}$. 
The map $D^r_P(\overline{j},i)$ induces the (place) polarization operator that sends the monomial $\chi_{lk}$ to the sum $\sum_{t=1}^{\binom{s}{r}} \chi_{lk_t}$.

It is clear from the definition that all polarization operators commute with each other.

The polarization operators act on symmetrizers $T[k:l]$ of shape $\lambda$ in the following simple way.

If $D^r_L(\overline{j},i)(T_k)=\sum_{t=1}^{z_U} T_{k_t}$, then

\[D^r_L(\overline{j},i)T[k:l]=\sum_{t=1}^{z_U}\sum_{\kappa\in C(T), \rho\in R(T)} sgn(\kappa) \chi_{k_t,l*\kappa*\rho}=\sum_{t=1}^{z_U} T[k_t:l].\]

Analogously, if $(T_l)D^r_L(\overline{j},i)=\sum_{t=1}^{z_U} T_{l_t}$, then
\[T[k:l]D^r_P(\underline{j},i)=\sum_{t=1}^{z_V}\sum_{\kappa\in C(T), \rho\in R(T)} sgn(\kappa) \chi_{k*\rho*\kappa,l_t}=\sum_{t=1}^{z_V} T[k:l_t]. \]

For an uncolored tableau $T_k$, denote by $c_i(T_k,j)$ the number of occurences of the symbol $i$ in the $j$th column of $T_k$
and by $r_i(T_k,j)$ the number if occurences of the symbol $i$ in the $j$th row of $T_k$.

Fot the uncolored bitableau $(T_k,T_l)$ of shape $\lambda$, we define
\[C_L(T_k)=\prod_{1\leq j\leq \lambda_1} \prod_{1\leq i\leq m+n} D_L^{c_i(T_k,j)}(\overline{j},i),\]
\[C_P(T_l)=\prod_{1\leq j\leq \lambda'_1} \prod_{1\leq i\leq m+n} D_P^{r_i(T_L,j)}(\underline{j},i),\]
and the Capelli operator $C(T_k,T_l)$ of the bitableau $(T_k,T_l)$ as
\[C(T_k,T_l)=C_L(T_k)\circ C_P(T_l).\]

In light of the previous remarks on effect of polarization operators on symmetrizers, if $C_L(T_k)(T_r)=\sum_t T_{k_t}$ and $(T_s)C_P(T_l)=\sum_u T_{l_u}$, then 
\begin{equation}\label{(1)}
C_L(T_k) T[r:s]=\sum_t T[k_t:s], \qquad T[r:s]C_P(T_l)=\sum_u T[r:l_u]
\end{equation}
and 
\[C(T_k,T_l)T[r:s]=\sum_t\sum_u T[k_t:l_u].\]

Denote by $T_{\overline{\ell}(\lambda)}$ the tableau of shape $\lambda$ such that its $j$th column consists solely of odd symbols $\overline{j}$ for each $j$, and denote by $T_{\underline{\ell}(\lambda)}$ the tableau of shape $\lambda$ such that its 
$j$th row consists solely of even symbols $\underline{j}$ for each $j$.

\subsection{$C(T_k,T_l)T[k:l]\neq 0$ for semistandard $T_k,T_l$.}

\begin{pr}\label{p4.1}
Let $T_k,T_l$ be semistandard tableaux of the shape $\lambda$. Then 
$C_L(T_k)T[k:l]=T[\overline{\ell}(\lambda):l]\neq 0$, $T[k:l]C_P(T_l)=T[k:\underline{\ell}(\lambda)]\neq 0$ and $C(T_k,T_l)T[k:l]=T[\overline{\ell}(\lambda):\underline{\ell}(\lambda)]\neq 0$.
\end{pr}
\begin{proof}
Denote by $T_{k_1}, \ldots, T_{k_z}$ all possible tableaux corresponding to summands of $C_L(T_k)(T_k)$. One of the summands is $T_{\overline{\ell}(\lambda)}$ and every $T_{k_t}$ has the same content as $T_{\overline{\ell}(\lambda)}$. 

Among the tableaux of shape $\lambda$ and content identical to that of $T_{\overline{\ell}(\lambda)}$, the tableau 
$T_{\overline{\ell}(\lambda)}$ is the maximal element with respect to the preorder $\unlhd_c$. Therefore, 
if $T_{k_t}\neq T_{\overline{\ell}(\lambda)}$, then $T_{k_t}\lhd_c T_{\overline{\ell}(\lambda)}$.
By Proposition \ref{p3.2}, if $T_{k_t}\neq T_{\overline{\ell}(\lambda)}$, then the symmetrizer $T[k_t:l]$ is a linear combination of symmetrizers $T[q:l]$, where $T_q$ is semistandard and $T_q\unlhd_c T_{k_t}\lhd_c  T_{\overline{\ell}(\lambda)}$. Since the only semistandard tableau of the same shape and content as $T_{\overline{\ell}(\lambda)}$ is $T_{\overline{\ell}(\lambda)}$ itself, we conclude that each $T[q:l]=0$ 
and $T[k_t:l]=0$ for $T_{k_t}\neq T_{\overline{\ell}(\lambda)}$. Therefore, 
$C_L(T_k)T[k:l]= T[\overline{\ell}(\lambda):l]$.

Analogously, denote by $T_{\underline{\ell}(\lambda)}=T_{l_1}, \ldots, T_{l_z}$ all possible tableaux corresponding to summands of $C_P(T_l)(T_l)$. Every summand $T_{l_u}$  has the same content as $T_{\underline{\ell}(\lambda)}$. 
Among the tableaux of shape $\lambda$ and the same content as $T_{\underline{\ell}(\lambda)}$, 
the tableau $T_{\underline{\ell}(\lambda)}$is the maximal element with respect to the preorder $\unlhd_r$.
Therefore, $T_{l_u}\neq T_{\underline{\ell}(\lambda)}$ implies $V^k\lhd_r \underline{T}_{\lambda}$.
By Proposition \ref{p3.1}, the symmetrizer $T[k:l_u]$ is a linear combination of symmetrizers $T[k:q]$, where $T_q$ is semistandard and $T_q\lhd_r T_{\underline{\ell}(\lambda)}$. Since the only semistandard tableau of the same shape and content as $T_{\underline{\ell}(\lambda)}$ is $T_{\underline{\ell}(\lambda)}$ itself, we conclude that each 
$T[k:q]=0$ and 
$T[k:l_u]=0$ for $T_{l_u} \neq T_{\underline{\ell}(\lambda)}$. Therefore 
$T[k:l]C_P(T_l)=T[k:T_{\underline{\ell}(\lambda)}]$.

Combining the above, we obtain $C(T_k,T_l)T[k:l]=T[\overline{\ell}(\lambda):\underline{\ell}(\lambda)]\neq 0$.
\end{proof}

\subsection{Triangularity condition for Capelli operators}

We say that a tableau $T_k$ is row-injective, if it does not contains two identical odd entries in the same row.
We say that a tableau $T_k$ is column-injective, if it does not contains two identical even elements in the same column.

The following triangularity condition is a modification of Theorem 4.4 of \cite{cl1}.

\begin{pr}\label{p4.2}
Assume that $T_k,T_l,T_i,T_j$ are semistandard tableaux of shape $\lambda$.
If $C_L(T_k)T[i:j]\neq 0$, then $T_k\unlhd_c T_i$. If $T[i:j]C_P(T_l)\neq 0$, then $T_l\unlhd_r T_j$. 
Thus $C(T_k,T_l)T[i:j]\neq 0$ implies $T_k\unlhd_c T_i$ and $T_l\unlhd_r T_j$. 
\end{pr}
\begin{proof}
We show that $C_L(T_k)T[i:j]\neq 0$, then $T_k\unlhd_c T_i$. 
For an ordered pair $(q,p)$ of natural numbers we define the `$(q,p)$' part of $C_L(T_k)$ as
\[C_{qp}(T_k)=\prod_{j\leq q}\prod_{i\leq p} D_L^{c_i(T_k,j)}(\overline{j},i).\]
Since all polarization operators commute, the assumption  $C_L(T_k)T[i:j]\neq 0$ implies
$C_{qp}(T_k)T[i:j]\neq 0$ for all $(q,p)$.

Using (\ref{(1)}) and Lemma \ref{l3} we infer that only symmetrizers $T[k_t:j]$ of row-injective tableaux $T_{k_t}$ will contribute non-zero terms to $C_{qp}(T_k)T[i:j]$. Therefore the tableau $T_k$ has the following property:
It is possible to find simultaneously for all pairs $(i,j)$, where $i\leq q$ and $j\leq p$, $c_i(T_k,j)$-times the
letter $i$ in $T_k$ such that, after having replaced these is by $\overline{j}$s, we obtain a row-injective tableau $T_{qp}$.

Word-by-word repetition of the arguments from the proof of Theorem 4.4 of \cite{cl1} implies $T_k\unlhd_c T_i$.

The statement $T[i:j]C_P(T_l)\neq 0$ implies $T_l\unlhd_r T_j$ is proved symmetrically using Lemma \ref{l1} instead of Lemma \ref{l3}, and utilizing column-injective tableaux instead of row-injective tableuax. 

The last implication is now obvious.
\end{proof}

\subsection{Linear independence of $T[i:j]$ for semistandard $T_i,T_j$}

\begin{pr}\label{p4.3}
The symmetrizers $T[i:j]$ for semistandard tableaux $T_i,T_j$ of shape $\lambda$ are linearly independent.
\end{pr}
\begin{proof}
Recall that $T_i\unlhd_c T_j$ if and only if $T_j\unlhd_r T_i$.
Also, there is a linear order $\leq $ on the set of semistandard tableaux of shape $\lambda$ that refines column-dominance preorder $\unlhd_c$, and its opposite order refines the row-dominance preorder $\unlhd_r$. 

Assume there is a nontrivial dependency relation 
\[\sum_{T_i,T_j \text{ semistandard}} c_{T_iT_j}T[i:j]=0.\]
Choose $T_i$ minimal such that $c_{T_iT_l}\neq 0$ for some $T_l$, and choose $T_j$ maximal such that $c_{T_iT_j}\neq 0$.
By Proposition \ref{p4.2}, if we apply the Capelli operator $C(T_i,T_j)$ to the above dependency relation, we obtain 
$c_{T_iT_j}C(T_i,T_j)T[i:j]=0$. Using Proposition \ref{p4.1} we infer 
$c_{T_iT_j}T[\overline{\ell}(\lambda):\underline{\ell}(\lambda)]=0$, which is a contradiction. 
\end{proof}

\subsection{The $S$-superbimodule $A_{\lambda}$}

By Theorem 12.1 of \cite{brini}, we have the decomposition \[A(m|n,r)=\bigoplus_{\lambda \vdash r} A_{\lambda}\] as $S$-superbimodules. Therefore, to understand the $S$-superbimodule structure of $A(m|n,r)$, it is enough to describe the structure of the component $A_{\lambda}$.

Recall the definition of tableau $T_\ell$ from \cite{mz}: Each row of $T_{\ell}$ of index $1\leq i\leq m$ is filled with identical entries $i$. The remaining part of the tableau is such that its $j$th column is filled with entries $m+j$ for $1\leq j\leq n$.
Since $\lambda$ is an $(m|n)$-hook partition, this described the tableau $T_\ell$ uniquely.

\begin{lm}
\label{l5} For a fixed $i$, the $K$-space spanned by $T[i:j]$ for
all $j$ is isomorphic to the left supermodule $D_{\lambda}$. Analogously,
for a fixed $j$, the $K$-space spanned by $T[i:j]$ for all $i$ is
isomorphic to the right supermodule $D_{\lambda}^{}$.
\end{lm}

\begin{proof}
Recall the bideterminant $T^{\lambda}(i:j)$ from
Definition 3.3. of \cite{mz}. Since $T[\ell:j]$ is a nonzero scalar multiple of 
$T^{%
\lambda}(\ell:j)$, Proposition 3.5 of \cite{mz} implies that $D_{\lambda}$ has
a basis consisting of $T[\ell:j]$ for semistandard $T_j$.

The $K$-space spanned by $T[i:j]$ for fixed $i$ is a left $S$%
-supermodule with the highest vector $\chi_{i,\ell}$ of weight $\lambda$.
Since by Proposition \ref{p2}, its dimension is the same as that of $D_{\lambda}$%
, it is isomorphic to $D_{\lambda}$ as a left $S$-supermodule.

Analogously, $D_{\lambda}^{o}$ has a basis consisting of $T[i:\ell]$
for semistandard $T_i$, and the $K$-space spanned by $T[i:j]$ for
fixed $j$ is a right $S$-supermodule with the highest vector $\chi_{\ell,j}$
of weight $\lambda$ that is isomorphic to $D_{\lambda}^{o}$ as a right $S$%
-supermodule.
\end{proof}

\begin{tr}
\label{t4.1} The space $A_{\lambda}$ is a $S$-superbimodule and it has a
basis consisting of symmetrizers $T^{\lambda}[i:j]$ for $T_i$, $T_j$
semistandard. The space $A_{\lambda}$ is isomorphic to the $S$%
-superbimodule $D_{\lambda}\otimes D^o_{\lambda}$.
\end{tr}
\begin{proof}
Theorem \ref{t3.1} states that every $T^{\lambda}[i:j]$ is a $\mathbb{Q}$-linear
combination of $T^{\lambda}[k:l]$ for $T_k, T_l$ semistandard.

Proposition \ref{p4.1} shows that the symmetrizers $T^{\lambda}[k:l]$ for $T_k, T_l$ semistandard are linearly
independent.

The isomorphism $A_{\lambda}\simeq D_{\lambda}^o\otimes D_{\lambda}$ as a $S$%
-superbimodule follows from Lemma \ref{l5}.
\end{proof}

\section{Modular reduction of symmetrizers}

Let us describe the process of the modular reduction for Schur superalgebras.
Assume that $K$ is a field of positive characteristic different from 2.
The algebra $S(m|n,r)_{\mathbb{Q}}$ defined over the ground field $\mathbb{Q}$ has a basis 
consisting of elements $\xi_{ij}$ for $i,j,\in I(m|n,r)$. Denote by $S(m|n,r)_{\mathbb{Z}}$ 
the $\mathbb{Z}$-module generated by these elements. Then $S(m|n,r)_{\mathbb{Z}}$ is multiplicatively
closed, and $S(m|n,r)_K=S(m|n,r)_{\mathbb{Z}}\otimes_{\mathbb{Z}} K$ is the $K$-algebra with $K$-basis 
consisting of elements $\xi_{i,j}\otimes 1_K$.

By a $\mathbb{Z}$-form of the superbimodule $V_{\lambda,\mathbb{Q}}$ we mean a 
nonzero $\mathbb{Z}$-
subsupermodule $V_{\lambda,\mathbb{Z}}$ 
of $V_{\lambda,\mathbb{Q}}$ that is closed under the left and right action of $S(m|n,r)_{\mathbb{Z}}$. 

If $V_{\lambda,\mathbb{Z}}$ is a $\mathbb{Z}$-form of $V_{\lambda,\mathbb{Q}}$, then 
the superbimodule $V_{\lambda, K}=V_{\lambda, \mathbb{Z}}\otimes K$ over $S(m|n,r)_K$ is said to be obtained by
modular reduction from $V_{\lambda,\mathbb{Q}}$. An important property of the modular reduction is that 
regardless of the choice of a $\mathbb{Z}$-form $V_{\lambda, \mathbb{Z}}$, the multiplicities of simple 
modules as composition factors of $V_{\lambda, K}$ remain the same.

\subsection{Modified symmetrizers}

We are going to modify our previous definition of symmetrizers.
For a tableau $T_j$, denote by $m_{ir}(j)$ the number of occurences of the
symbol $1\leq i\leq m$ in the $r$th row of $T_i$, and by $n_{id}(j)$ the number of
occurences of the symbol $m+1\leq i\leq m+n $ in the $d$-th column of $T_j$. Further, define 
\[r(T_j)=\prod_{i=1}^m\prod_r m_{ir}(j)! \text{ and } c(T_j)=\prod_{i=m}^{m+n}\prod_d  n_{i,d}(j)!.\]

\begin{df}
Let $\lambda$ be a partition of $r$ and $T$ be a basic tableau of shape $\lambda$.
For $i,j\in I(m|n,r)$ define the {\it modified symmetrizer}
\begin{equation}
\label{5}
T\{i:j\}=\frac{1}{r(T_i)c(T_j)}T[i:j].
\end{equation}
\end{df}

If $\kappa\in C(T)$ permutes only identical entries of $T_j$ from the set $\{m+1, \ldots, m+n\}$, then  $\chi_{i,j*\kappa}=\chi_{i,j}$. 
If $\rho\in R(T)$ permutes only identical entries of $T_i$ from the set $\{1, \ldots, m\}$, then  $\chi_{i*\rho,j}=\chi_{i,j}$. 
Therefore, all coefficients in $T\{i:j\}$ are 
integers.

\subsection{$\mathbb{Z}$-form of $A_{\lambda, \mathbb{Q}}$}

Let us recall some concepts from \cite{bk,lz}; 
for details consult these sources.
Let ${\rm Dist}({\rm GL}(m|n))$ be an algebra of distributions of the general linear supergroup ${\rm GL}(m|n)$.
The Kostant $\mathbb{Z}$-form of $U(\mathfrak{g})_{\mathbb{Z}}$ is defined as a
$\mathbb{Z}$-subalgebra of $U(\mathfrak{g})_{\mathbb{C}}$ generated by elements 
\[\binom{e_{qq}}{k}=\frac{e_{qq}(e_{qq}-1)\dots (e_{qq}-k+1)}{k!}\]
for all $k>0$ together with elements $e_{pq}^{(t)}=\frac{e_{pq}^t}{t!}$, where $p\neq q$ for all $t>0$ in the case $|q|+|p|$ is even, 
and $e_{pq}$, where $p\neq p$ in the case $|q|+|p|$ is odd. 
Here $e_{pq}\in {\rm Dist}_1({\rm GL}(m|n))$ is 
such that $e_{pq}(\chi_{ij})=\delta_{pi}\delta_{qj}$ and $e_{pq}(1)=0$.

The left superderivation $D_{pq}$ and the right superderivation $_{pq}D$ of $A(m|n)$ are defined by
$D_{pq}(c_{kl})=\delta_{qk}c_{pl}$ and $(c_{kl})_{pq}D=\delta_{lp}c_{kq}$, respectively.

For $i\in I(m|n,r)$ and integers $1\leq p \neq q \leq m+n$, denote 
by $i|_{p\stackrel{t}{\to} q}$ an element of $I(m|n,r)$ that is obtained 
from $i$ by replacing exactly $t$ occurrences of $p$ in $i$ by $q$.

\begin{lm}\label{l42}
For any $i,j\in I(m|n,r)$ and any $\pi\in S_r$, the following formulae hold:
\[(\chi_{i, j*\pi}){_{p q}}D^{(t)}=\sum_{j|_{p\stackrel{t}{\to} q}}
(-1)^{s(j|_{p\stackrel{t}{\to} q} , j)}\chi_{i, j|_{p\stackrel{t}{\to}
q}*\pi}, {\rm and}\]
\[D^{(t)}_{pq} (\chi_{i*\pi, j})=\sum_{i|_{q\stackrel{t}{\to} p}}(-1)^{s(i|_{q\stackrel{t}{\to} p} , i)}
\chi_{i|_{q\stackrel{t}{\to} p}*\pi, j}.\]
\end{lm}
\begin{proof} The first formula is taken from Lemma 4.2 of \cite{mz}. The second formula is proved analogously as in the proof of Lemma 4.2 of \cite{mz}.
\end{proof}

\begin{tr} \label{t5.1}
If $\lambda$ is a $(m|n)$-hook partition of $r$, then 
the $\mathbb{Z}$-span $A_{\lambda,\mathbb{Z}}$ of modified symmetrizers $T\{i:j\}$ is 
a $\mathbb{Z}$-form of the $S(m|n,r)$-superbimodule $A_{\lambda,\mathbb{Q}}$ under the left and right actions of $S(m|n,r)_{\mathbb{Z}}$.
\end{tr}
\begin{proof} Modifying the proof of Theorem 1 of \cite{mz}, we write
\[\begin{aligned}&T\{i :j\}{_{p q}}D^{(t)}=\\
&\sum\limits_{\substack{0\leq l_1\leq n_{j,1, p}, \ldots , 0\leq l_{\lambda_1}\leq n_{j,\lambda_1, p} \\
l_1+\ldots +l_{\lambda_1}=t}}
\prod_{1\leq k\leq\lambda_1}(-1)^{s(j|_{p\stackrel{t}{\to} q} , j)}\binom{n_{j,i,q}+l_k}{l_k}T\{i : j(l_1 ,\ldots , \ l_{\lambda_1})\},
\end{aligned}\]
where 
$T_{j(l_1 ,\ldots , \ l_{\lambda_1})}$ is obtained from $T_j$ by
replacing the $l_k$ topmost occurrences of $p$ in $i$-th column of
$T_j$ by $q$ for $1\leq k\leq\lambda_1$.

Analogously, 
\[\begin{aligned}&D^{(t)}_{p q}T \{i :j\}=\\
&\sum\limits_{\substack{0\leq l_1\leq n_{j,1, p}, \ldots , 0\leq l_{\lambda'_1}\leq n_{j,\lambda'_1, p} \\
l_1+\ldots +l_{\lambda'_1}=t}}
\prod_{1\leq k\leq\lambda'_1}(-1)^{s(i|_{q\stackrel{t}{\to} p} , i)}\binom{m_{i,j,q}+l_k}{l_k}T\{i(l_1 ,\ldots , \ l_{\lambda'_1}):j\},\end{aligned}\]
where 
$T_{i(l_1 ,\ldots , \ l_{\lambda'_1})}$ is obtained from $T_i$ by
replacing the $l_k$ leftmost occurrences of $q$ in $i$-th row of
$T_i$ by $p$ for $1\leq k\leq\lambda'_1$.

Since \[T\{i:j\}\binom{_{qq}D}{k}=\binom{cont(j,q)}{k}T\{i:j\}\]
and \[\binom{D_{qq}}{k}T\{i:j\}=\binom{cont(i,q)}{k}T\{i:j\},\] where $cont(j,q)$ and $cont(i,q)$ denote the number of appearances of $q$ in $j$ and $i$, respectively, the claim follows.
\end{proof}

\subsection{$S$-superbimodule filtration of $A_{\lambda,K}$ for $char K=p>2$}

Once we have the $\mathbb{Z}$-form $A_{\lambda,\mathbb{Z}}$, we can apply the modular reduction and define
$A_{\lambda,K}=A_{\lambda,\mathbb{Z}}\otimes K$ for a ground field $K$ of characteristic $p>2$.

The left supermodule $D_{\lambda}$ has a $\mathbb{Z}$-form $D_{\lambda, \mathbb{Z}}$ consisting of modified symmetrizers $T\{\ell:j\}$. Analogously, the right supermodules $D^o_{\lambda}$ has a $\mathbb{Z}$-form $D^o_{\lambda, \mathbb{Z}}$ consisting of modified symmetrizers $T\{i:\ell\}$.
Applying modular reduction, we get $D_{\lambda,K}=D_{\lambda,\mathbb{Z}}\otimes K$ and 
$D^o_{\lambda,K}=K\otimes D^o_{\lambda,\mathbb{Z}}$.

\begin{lm}
The $S$-superbimodule $A_{\lambda,K}$ has a filtration by $S$-superbi-modules $L_S(\kappa)\otimes L^o_S(\mu)$, 
where $L_S(\kappa)$ is the irreducible left $S$-supermodule of the highest weight $\kappa$ linked to $\lambda$, 
and $L^o_S(\mu)$ is the irreducible right $S$-supermodule of the highest weight $\mu$ linked to $\lambda$. 
\end{lm}
\begin{proof}
The left $S$-supermodule $D_{\lambda,K}$ has a filtration by left $S$-super-modules
\[0\subset N_1 \subset N_2 \subset \ldots \subset N_s=D_{\lambda,K}\]
where $N_1=L_S(\mu_1)$, and 
$N_j/N_{j-1}\simeq L_S(\mu_j)$ for $j=2, \ldots, s$ are irreducible left $S$-modules of highest weight $\mu_j$  linked to $\lambda$, since one of the weights $\mu_j$ equals $\lambda$.

Analogously, the right $S$-supermodule $D^o_{\lambda,K}$ has a filtration by right $S$-supermodules
\[0\subset M_1 \subset M_2 \subset \ldots \subset M_t=D^o_{\lambda,K}\]
where $M_1=L^o_S(\kappa_1)$ and 
$M_i/M_{i-1}\simeq L^o_S(\kappa_i)$ for $i=2, \ldots, t$ are irreducible right $S$-modules of highest weight $\kappa_i$ linked to $\lambda$, since one of the weights $\kappa_i$ equals $\lambda$.

Then $A_{\lambda, K}$ has a filtration such that each of its $st$ factors is a  $S$-superbimodule $L_S(\kappa_i)\otimes L^o_S(\mu_j)$ for 
$i=1, \ldots, t$ and $j=1, \ldots, s$. 
\end{proof}

\subsection{Straightening of modified symmetrizers}

Motivated by Theorem \ref{t5.1},  it is natural to ask if modified symmetrizers are $\mathbb{Z}$-linear combinations of modified symmetrizers for $T_i,T_j$ semistandard. The anwer is affirmative.
\begin{tr}\label{t5.2} 
Let $\lambda$ be an $(m|n)$-hook partition. Then every modified symmetrizer $T^{\lambda}\{i:j\}$ is a $\mathbb{Z}$-linear combination of modified symmetrizers $T^{\lambda}\{u:v\}$ for $T_u, T_v$ semistandard.
\end{tr}
\begin{proof}
We need to review and modify the proof of Proposition \ref{p1}. We adhere to the notation of the proof of Proposition \ref{p1}.
First consider the case $\underline{j_{a_q}>j_{b_q}}$:

We have 
\begin{equation}\label{3}
T[i:j]=-\sum_{t=2}^l sgn(\sigma_t)T[i:j*\sigma_t],
\end{equation}
where $T_{j\sigma_t}<T_j$ for each $t>1$.

We group all symmetrizers $T[i:j*\sigma_t]$ with the same $\tilde{j}=j\sigma_t$. Since $\sigma_t$ in the same group only differ by a permutation of identical odd elements, we obtain $j*\sigma_t=sgn(\sigma_t)j$.

Recall that $n_{kd}(j)$ is the number of
occurences of the symbol $m+1\leq k\leq m+n $ in the $d$-th column of $T_j$.
Since each $\sigma_t$ only permutes elements in the $d$ and $d+1$st column of $T_j$, we need to be concerned only 
with $n_{kd}(\tilde{j})$ and $n_{k,d+1}(\tilde{j})$ for $m+1\leq k\leq m+n$.

Clearly, $n_{k,d}(\tilde{j})+n_{k,d+1}(\tilde{j})= n_{k,d}(j)+n_{k,d+1}(j)$ for each $k$. There are
\[\prod_{k=m+1}^{m+n}\binom{n_{k,d}(j)+n_{k,d+1}(j)}{n_{k,d}(\tilde{j})}\]
elements $\sigma_t$ such that $j\sigma_t=\tilde{j}$.
Then
\[\begin{aligned}&\prod_{k=m+1}^{m+n}n_{k,d}(\tilde{j})!n_{k,d+1}(\tilde{j})!\binom{n_{k,d}(j)+n_{k,d+1}(j)}{n_{k,d}(\tilde{j})}\\
&=\prod_{k=m+1}^{m+n} (n_{k,d}(j)+n_{k,d+1}(j))!=\prod_{k=m+1}^{m+n}n_{k,d}(j)!n_{k,d+1}(j)!
\end{aligned}\]
because at least one of $n_{k,d}(j)$ and $n_{k,d+1}(j)$ vanish.

Therefore, when we replace symmetrizers in (\ref{3}) by modified symmetrizers, and divide by $r(T_i)c(T_j)$, we obtain
\[T\{i:j\}=-\sum_{\tilde{j}\neq j} sgn(\sigma_t)T\{i:\tilde{j}\},\]
where $T_{\tilde{j}}<T_j$.

Case \underline{$j_{a_q}=j_{b_q}>m$:}

We need to modify the notation as follows:
\[ T_j:
\begin{array}{ccc}
C_d&&C_{d+1}\\
&&\\
j_{a_1}&&j_{b_1}\\
\ldots&&\ldots\\
\ldots&&j_{b_{p-1}}\\
&&\wedge\\
\ldots&&j_{b_p}\\
&&\vert\vert\\
\ldots&&\ldots\\
&&\vert\vert\\
j_{a_q}&=&j_{b_q}\\
\vert\vert&&\vert\vert\\
\ldots&&\ldots\\
\vert\vert&&\\
j_{a_r}&&\vert\vert\\
\wedge&&j_{b_t}\\
j_{a_{r+1}}&&\wedge\\
\ldots&&j_{b_{t+1}}\\
\ldots&&\ldots\\
j_{a_s}&&\\
\end{array},
\]
where there is no relation between indices $r$ and $t$.

Define $X=\{a_s, \ldots, a_q\}$ and $Y=\{b_t, \ldots, b_1\}$. By Lemma \ref{l2}
we obtain 
\begin{equation}\label{4}\sum_{t=1}^l sgn(\sigma_t)T[i:j*\sigma_t]=0.
\end{equation}

Again we group together all symmetrizers $T]i:j*\sigma_t]$ with the same $\tilde{j}=j\sigma_t$. Since $\sigma_t$ in the same group only differ by a permutation of identical odd elements, we get all corresponding expressions $sgn(\sigma_t)j*\sigma_t=\epsilon_{\tilde{j}}\tilde{j}$, where $\epsilon_{\tilde{j}}=\pm 1$.

Again, $n_{k,d}(\tilde{j})+n_{k,d+1}(\tilde{j})= n_{k,d}(j)+n_{k,d+1}(j)$ for each $k$. There are
\[\prod_{k=m+1}^{m+n}\binom{n_{k,d}(j)+n_{k,d+1}(j)}{n_{k,d}(\tilde{j})}\]
elements $\sigma_t$ such that $j\sigma_t=\tilde{j}$.
Then
\[\begin{aligned}&\prod_{k=m+1}^{m+n}n_{k,d}(\tilde{j})!n_{k,d+1}(\tilde{j})!\binom{n_{k,d}(j)+n_{k,d+1}(j)}{n_{k,d}(\tilde{j})}\\
&=\prod_{k=m+1}^{m+n} (n_{k,d}(j)+n_{k,d+1}(j))!\\
&\prod_{k=m+1}^{m+n}n_{k,d}(j)!n_{k,d+1}(j)!\binom{n_{k,d}(j)+n_{k,d+1}(j)}{n_{k,d}(j)},
\end{aligned}\]

Therefore, when we replace symmetrizers in (\ref{4}) by modified symmetrizers, and divide by 
\[r(T_i)c(T_j)\prod_{k=m+1}^{m+n}\binom{n_{k,d}(j)+n_{k,d+1}(j)}{n_{k,d}(j)},\]
we obtain
\[T\{i:j\}=-\sum_{\tilde{j}\neq j} \epsilon_{\tilde{j}} T\{i:\tilde{j}\},\]
where $T_{\tilde{j}}<T_j$.

The modification to Proposition \ref{p2} are analogous.
\end{proof}

The previous theorem generalizes Proposition 4.9 of \cite{mz}. 

\subsection{Connection of $A_\lambda$ to factors of Donkin-Koppinen filtration of $K[G]$}

Using a process of modular reduction and Theorems \ref{t5.1} and \ref{t5.2}, we obtain that over a field of characteristic $p>2$, the superbimodule $A_{\lambda, K}$ has a basis consisting of modified symmetrizers $T^\lambda \{i:j\}$. 
This result is related to our previous work \cite{mz}. 

Since the modular reduction preserves the character, we obtain that, for polynomial weights $\lambda$, the
characters of $A_{\lambda}$ and $D_{\lambda}\otimes D_{\lambda}^o$ coincide.

From the left regular representation of $S(1|1,r)$ for $r$ divisible by the characteristic $p$ given in \cite{mz-1},
we get
\[A(1|1,r)_S= \begin{array}{cccccccccccc}1&&&1&&&&&r-1&&&r-1\\
0&\oplus&0&&2&\oplus&\ldots&r-2&&r&\oplus&r\\
&&&1&&&&&r-1&&&
\end{array}\]
We observe that for all weights $\lambda\neq 1^r$ (those include all strong polynomial weights) we have 
$A_{\lambda}\simeq \nabla_S(\lambda)\otimes \Delta_S(\lambda)^*$, while for $\lambda=1^r$ which is not strong polynomial, we have 
\[A_{1^r}\simeq \nabla_S(1^r)^*\oplus \Delta_S(1^r)^*\simeq \begin{array}{ccc}1&&0\\0&\oplus&1\end{array}\not \simeq \begin{array}{ccc}0&&0\\1&\oplus&1\end{array}\simeq \nabla_S(1^r)\otimes \Delta_S(1^r)^* \]
as right $S$-supermodules.

Motivated by this example we ask the following question:
\begin{qs}\label{question}
For which $(m|n)$-hook partitions $\lambda$ is there an isomorphism
$A_{\lambda,K} \simeq  \nabla_S(\lambda)\otimes \Delta_S(\lambda)^*$ of $S$-superbimodules? 
\end{qs}

If $\lambda_m\geq n$, then the largest polynomial subsupermodule $\nabla(\lambda)$ of $H^0_G(\lambda)$ is $H^0_G(\lambda)$, the largest polynomial factorsupermodule $\Delta(\lambda)$ of the Weyl supermodule $V_G(\lambda)$ is $V_G(\lambda)$, and
$\nabla(\lambda)\otimes \Delta(\lambda)^*  \simeq H^0_G(\lambda)\otimes V_G(\lambda)^*$.
Therefore, if $\lambda$ satisfies the conditions in Question \ref{question} and $\lambda_m\geq n$, 
then $A_{\lambda, K}$ describes the $G$-superbimodule structure of the factor of Donkin-Koppinen filtration corresponding to such $\lambda$.

Of course, tensoring with a power of the Berezian $Ber$ of weight $\nu=(n, \ldots, n|-m,\ldots , -m)$ extends the description of the $G$-superbimodule structure for factors of Donkin-Koppinen filtration corresponding to dominant $\lambda$ such that $\mu=\lambda + k\nu$ satisfies the conditions in Question \ref{question}.


\begin{thebibliography}{9}

\bibitem{br} A.Berele and A.Regev, {\em Hook Young diagrams with
applications to combinatorics and to representations of Lie
superalgebras}, Adv. Math., 64 (1987), 118--175.

\bibitem{brini} A. Brini, {\em Combinatorics, superalgebras, invariant theory and representation theory}, Seminaire Lotharingien de Combinatoire 55 (2007), Article B55g.

\bibitem{bk} J. Brundan and J. Kujawa, {\em A new proof of the Mullineux Conjecture},
J. Algebraic Combin. 18 (2003), 13--39.

\bibitem{cl} R.W.Carter, G.Lustig, \emph{On the modular representations of the general linear and symmetric groups}, 
Math. Z. 136 (1974), 193--242.

\bibitem{cl0} M.Clausen, {\em Letter place algebras and characteristic-free approach to the representation theory of the general linear and symmetric groups I,II.}, Adv. Math. 33 (1979), 161--191 and 38 (1980), 152--177.

\bibitem{cl1} M.Clausen, \emph{Multivariate polynomials, standard tableaux, and representations of symmetric groups},
J. Symbolic Computation 11 (1991), 483--522.

\bibitem{cl2} M.Clausen, \emph{Dominance order, Capelli operators, and straightening of bideterminants}, 
Europ. J. Combinatorics  5 (1984), 207--222.

\bibitem{cps} E.Cline, B.Parshall and L. Scott, {\em Finite dimensional algebras and highest weight categories},
J. Reine Angew. Math. 391 (1988), 85--99.

\bibitem{dkr} J.Desarmenien, J.P.S.Kung and G.-C.Rota, {\em Invariant theory, Young bitableaux and combinatorics}, Adv. Math. 27 (1978), 63--92.

\bibitem{don} S.Donkin, {\em A filtration for rational modules}, Math. Z., 177 (1981), 1--8.


\bibitem{doub} S.R.Doubilet, G.-C.Rota and J. Stein, {\em On the foundation of combinatorial theory IX: combinatorial methods in invariant theory}, In Studies in Appl. Math. 53 (1974), 185--216.

\bibitem{jk} G.D.James and A.Kerber, \emph{The Representation Theory of the Symmetric Groups}, Encyclopedia of Mathematics 16, Cambridge University Press, Cambridge 1981. 

\bibitem{gr} J.A.Green, \emph{Polynomial Representations of $GL_n$}, Lecture Notes in Mathematics 830, Springer, Berlin 1980.

\bibitem{kac} V. G. Kac, {\em Representations of classical Lie superalgebras},  Lect. Notes Math., Vol. 676 (1978),
pp. 597--626.

\bibitem{kop} M.Koppinen, {\em Good bimodule filtrations for coordinate rings}, J. Lond. Math. Soc. (2) 30(2) (1984),
244--250.

\bibitem{lz} R. la Scala, A.N. Zubkov, {\em Costandard modules over Schur superalgebras in characteristic p}, J. Algebra Appl. 7 (2) (2008), 147--166.

\bibitem{lz2} R. la Scala, A.N. Zubkov, {\em Donkin–Koppinen filtration for general linear supergroups}, Algebr. Represent. Theor. (2012), 15:883--899.

\bibitem{m} F.Marko, \emph{Primitive vectors in induced supermodules for
general linear supergroups}, J. Pure Applied Algebra 219 (2015), 978--1007.

\bibitem{m1} F.Marko, \emph{Even-primitive vectors in induced supermodules for general linear
supergroups and in costandard supermodules for Schur superalgebras}, J. Algebr. Comb.
DOI 10.1007/s10801-019-00879-6.

\bibitem{mz-1}F.Marko and A.N.Zubkov, \emph{Schur superalgebras in characteristic $p$}, 
Algebra and Representation Theory 2006 9: 1--12.

\bibitem{mz0} F.Marko and A.N.Zubkov, \emph{Schur superalgebras in characteristic $p$, II},
Bull. London Math. Soc. {\bf38} (2006), 99--112.

\bibitem{mz} F.Marko and A.N.Zubkov, \emph{A note on bideterminants for
Schur superalgebras}, J. Pure Applied Algebra 215 (2011), 2223--2230.

\bibitem{martin} S. Martin, \emph{Schur algebras and representation theory},
Cambridge University Press 1983.

\bibitem{mu} N.J.Muir, \emph{Polynomial representations of the general
linear Lie superalgebras}, Ph.D. Thesis, University of London, 1991.

\bibitem{peel} M.H.Peel, \emph{Specht modules and symmetric groups}, J. Algebra 36 (1975), 88--97.

\bibitem{z} A.N.Zubkov, \emph{Some properties of general linear supergroups
and of Schur superalgebras}, (Russian), Algebra Logika, 45(2006), N 3,
257--299.
\end{thebibliography}
\end{document}